\documentclass[12pt,oneside]{article}
\usepackage[all]{xy}
\usepackage{amsmath,amssymb,amsfonts,amsthm,mathrsfs,dsfont,mathtools}
\usepackage[english]{babel}
\usepackage{graphicx,stackrel}
\usepackage{multirow,rotating,paralist}
\usepackage[toc,page,header]{appendix}
\DeclareMathAlphabet{\mathpzc}{OT1}{pzc}{m}{it}
\usepackage{minitoc}
\usepackage{extarrows}
\usepackage{hyperref}
\usepackage{csquotes}

 \small {\par\vskip8pt minus3pt\rm}
\newcounter{item}[section]
\newcounter{kirshr}
\newcounter{kirsha}
\newcounter{kirshb}

\newtheorem{theorem}{Theorem}[section]

\newtheorem{lemma}[theorem]{Lemma}

\newtheorem{proposition}[theorem]{Proposition}

\newtheorem{remark}[theorem]{Remark}

\newtheorem{definition}[theorem]{Definition}

\newcommand\undersym[2]{\raisebox{-6pt}{\tiny$#2$}{\kern-5pt}\mbox{$#1$}}
\newcommand\overcirc[1]{\raisebox{10pt}{\tiny$\circ$}{\kern-7pt}\mbox{$#1$}}

\title{\Large{Classification of 1-Weierstrass points on Kuribayashi quartics,\hspace*{-.12cm} II (with three parameters)}}
\date{}
\author{{\small Eslam E. Badr$^{1}$ and Mohammed A. Saleem$^{2}$}}

\begin{document}
\bibliographystyle{plain}
\maketitle                     
\vspace*{-.95cm}

\begin{center}
${^1}$Department of Mathematics, Faculty of Science, \\[-0.15cm]
Cairo University, Giza, Egypt \\[-0.15cm]
\vspace*{.3cm}
${^2}$Department of Mathematics, Faculty of Science,\\[-0.15cm]
Sohag University, Sohag, Egypt \\[-0.15cm]
\end{center}

\begin{center}
$
\begin{array}{ll}
\text{Emails:}&\text{eslam@sci.cu.edu.eg, eslam60145@yahoo.com} \\
  &\text{abuelhassan@yahoo.com}
\end{array}
$
\end{center}

\maketitle \bigskip
\begin{abstract}
In this paper, we classify the 1-Weierstrass points of the
Kuribayashi quartic curves with three parameters $a,\,b$ and $c$ defined
by the equation
\[
C_{a,b,c}:x^{4}+y^{4}+z^{4}+ax^{2}y^{2}+bx^{2}z^2+cy^{2}z^{2}=0,
\]
such that $(a^2-1)(b^2-1)(c^2-1)(a^2-4)(b^2-4)(c^2-4)(a^2+b^2+c^2-abc-4)\neq0.$ Furthermore, the geometry of these points is investigated.
\end{abstract}
\newpage

\section{Introduction}
Let $C_{a}$ be the smooth plane quartic curves defined by the
equation
\[
C_{a}:x^{4}+y^{4}+z^{4}+a(x^{2}y^{2}+x^{2}z^{2}+y^{2}z^{2})=0,\quad\quad
a\neq-1,\pm2.%
\]
These types of quartic curves are called \emph{Kuribayashi quartics
with one parameter. }Weierstrass points and automorphism groups of
Riemann surfaces of genus $3$ are studied in \cite{pa11, pa15}. In
particular, Weierstrass points and automorphism groups of $C_a$ are
studied in \cite{pa12}. Kuribayashi and his students used \emph{the
Wronskian method} to classify the number of the 1-Weierstrass points
of $C_a$. Alwaleed \cite{pa3, pa4} used the \emph{the Wronskian
method} together with the $S_4$ action on $C_a$ to classify the
number and investigate the geometry of the 2-Weierstrass points of
this family.

 Let $C_{a,b}$ be the smooth plane quartic curves defined by the equation
\[
C_{a,b}:x^{4}+y^{4}+z^{4}+ax^{2}y^{2}+b(x^{2}+y^{2})z^{2}=0,
\]
where $a$ and $b$ are two parameters such that
$(a^2-4)(b^2-4)(a^2-1)(b^{2}-a-2)\neq0.$ We call these quartic
curves \emph{Kuribayashi quartics with two parameter. }Hayakawa
\cite{pa5} investigated the conditions under which the number of the
Weierstrass points of $C_{a,b}$ is $12$ or $16$.
\par In \cite{pa20}, the present authors have classified the number of the 1-Weierstrass points of $C_{a,b}$ together with their geometry. Moreover, they have obtained the results of Kuribayashi and Sekita \cite{pa12} on 1-Weierstrass points of the quartics $C_a$ as a particular case.

 Let $C_{a,b,c}$ be the smooth plane quartic curves defined by the equation
\[
C_{a,b,c}:x^{4}+y^{4}+z^{4}+ax^{2}y^{2}+bx^{2}z^2+cy^{2}z^{2}=0,
\]
where $a,\,b$ and $c$ are three parameters such that \[(a^2-1)(b^2-1)(c^2-1)(a^2-4)(b^2-4)(c^2-4)(a^2+b^2+c^2-abc-4)\neq0.\]
Hayakawa \cite{pa2} investigated the conditions under which the number of the 1-Weierstrass points of $C_{a,b,c}$ is exactly $12$ or $<24$.

The aim of this paper is to generalize the results done in \cite{pa20}. For this purpose, we consider a different approach (group actions on Riemann surfaces) from that followed by Hayakawa to classify the number of such points completely. Furthermore, we investigate the geometry of these points.

The present paper is organized in the following manner. In section
$1,$ we present some preliminaries concerning the basic concepts
that will be used throughout the work \cite{pa8, pa13}. In section
$2,$ we establish our main results, \emph{Theorems 2.7, 2.8, 2.9, 2.10,} that
concern with the classification of the number of the 1-Weierstrass
points of the quartic curves $C_{a,b,c}$ together with their geometry.
In section $3$, we illustrate, through examples, the cases mentioned
in the main results. Finally, we conclude the paper with some
remarks, comments and related problems.
\section{Preliminaries}

\subsection{q-Weierstrass points}

Let $C$ be a smooth projective plane curve of genus $g\geq2$ and let $D$ be a
divisor on $C$ with $dim|D|=r\geq0$. We denote by $L(D)$ the $\mathbb{C}%
$-vector space of meromorphic functions $f$ such that $div(f)+D\geq0$ and
by $l(D)$ the dimension of $L(D)$ over $\mathbb{C}$. Then, the
notion of $D$-Weierstrass points \cite{pa13} can be defined in the following way:
\begin{definition}
Let $p\in C.$ If $n$ is a positive integer such that
\[
l\big(D-(n-1).p\big)>l\big(D-n.p\big),
\]
we call the integer $n$ a $D$-gap number at $p$.
\end{definition}

\begin{lemma}
Let $p\in C,$ then there are exactly $r+1$ $D$-gap numbers $\{n_{1},n_{2},...,n_{r+1}\}$ such that $n_{1}<n_{2}<...<n_{r+1}$. The sequence $\{n_{1},n_{2},...,n_{r+1}\}$ is called the $D$-gap sequence at $p$.
\end{lemma}

\begin{definition}
The integer
$\omega_{D}(p):=\sum_{i=1}^{r+1}(n_{i}-1)$ is called $D$-weight at $p$. If
$\omega_{D}(p)>0$, we call the point $p$ a $D$-Weierstrass point on $C$. In
particular, for the canonical divisor $K$, the $qK$-Weierstrass points $(q\geq1)$
are called $q$-Weierstrass points and the $qK$-weight is called $q$-weight and is denoted by $\omega^{(q)}(p)$.
\end{definition}

\begin{definition}
\em{\cite{pa4}} A point $p$ on a smooth plane curve $C$ is said to
be a flex point if the tangent line $L_p$ meets $C$ at $p$ with
contact order $I_{p}(C,L_{p})$ at least three. We say that $p$ is
$i$-flex, if $I_{p}(C,L_{p})-2=i$. The positive integer $i$ is
called the flex order of $p.$
\end{definition}

\begin{lemma}
\em{\cite{pa8}} Let $C: F(x,y,z)=0$ be a smooth projective plane
curve. A point $p$ on $C$  is a flex point if, and only if,
$H_F(p)=0,$ where $H_F$ is the Hessian curve of $C$ defined by
\[H_F:=det\left(
             \begin{array}{ccc}
             F_{xx} & F_{xy} & F_{xz} \\
             F_{yx} & F_{yy} & F_{yz} \\
              F_{zx} & F_{zy} & F_{zz}
                    \end{array}
                  \right)
.\]

\end{lemma}

\begin{lemma}
\em{\cite{pa14}} Let $C$ be a smooth projective plane quartic curve. The $1$-Weierstrass points on $C$ are nothing but flexes and divided into two types: ordinary flex and hyperflex points. Moreover, we have%

\begin{center}
    \begin{tabular}
[c]{|c|c|c|}\hline
$\omega^{(1)}(p)$ & $1$-gap Sequence & Geometry\\\hline
$1$ & $\{1,2,4\}$ & ordinary flex\\\hline
$2$ & $\{1,2,5\}$ & hyperflex\\\hline
\end{tabular}
\end{center}

\emph{\ }
\end{lemma}

\begin{lemma}
\em{\cite{pa7, pa13}} Let $C$ be a smooth projective plane curve of
genus $g$. The number of $q$-Weierstrass points $N^{(q)}(C),$
counted with their $q$-weights, is given by
\[
N^{(q)}(C)=\left\{
\begin{array}
[c]{lr}%
g(g^{2}-1),\,\,\,\,\,\,\,\,\,\,\,\,\,\,\,\,\,\,\,\,\,\,\,\,\,\,\,\,\,\,\,\,\,\,\,\,\,\,if\, q=1 & \\\\
(2q-1)^{2}(g-1)^{2}g,\,\,\,\,\,\,\,\,\,\,\,\,\,\,\,if\, q\geq2. &
\end{array}
\right.
\]
In particular, for smooth projective plane quartics $($i.e.\,\,$g=3)$, the number
of $1$-Weierstrass points, counted with their weights, is $24$ .
\end{lemma}

Let $W^{(q)}(C)$ be the set of $q$-Weierstrass points on $C$ and $G^{(q)}(p)$ the $q$-gap sequence at the point $p\in C$.

\begin{lemma} \em{\cite{pa3}}
Let $\tau$ be an automorphism on $C$, then we have
\[
\tau\big(W^{(q)}(C)\big)=W^{(q)}(C)\quad and\quad G^{(q)}(\tau(p))=G^{(q)}(p).
\]

\end{lemma}

\bigskip
\subsection{Group Action on Riemann Surfaces}
\vspace*{-.9cm}$\phantom{GroupActiononRiemannSurfaceaaaaaaaaaaaaaaaaaaaa}$[13]
\begin{definition}
An action of a finite group $G$ on a Riemann surface $C$ is a map
\[
\cdot:G\times C\longrightarrow C:(g,p)\longmapsto g\cdot p
\]
such that: $(gh)\cdot p=g\cdot(h\cdot p)$\,\,and\,\,$e\cdot p=p,$\,\,for all
$g,h\in G$ and $p\in C$, where $e$ is the identity element of $G.$
\end{definition}

\begin{definition}
The orbit of a point $p\in C$ is the set\,\, $Orb_{G}(p):=\{g\cdot p:g\in G\}.$
\end{definition}

\begin{definition}
The stabilizer of a point $p\in C$ is the subgroup
\[
G_{p}:=\{g\in G:g\cdot p=p\}.
\]
It is often called the isotropy subgroup of $p.$
\end{definition}

\begin{remark}
\em{The subgroup $G_{p}$ is cyclic and points in the same
orbit have conjugate stabilizers; Indeed, $G_{g\cdot
p}=gG_{p}g^{-1}.$ Moreover
\[|Orb_{G}(p)|\,|G_{p}|=|G|,\quad \forall p\in C.\]}\vspace*{.3cm}
\end{remark}
\noindent\textbf{Notation.}\,[1] The set of points $p\in C$ such
that $|G_{p}|>1$ is denoted by $X(C).$ Also
\[
X_{i}(C):=\{p\in C:|G_{p}|=i\}.
\]


\section{Main results}
A group action of order 4 on $C_{a,b,c}$ can be defined as follows. Let $H$ be the projective transformation group of order $4$ generated by the two elements $\sigma$ and $\tau$ of orders $2$, where
\[
\sigma:=\left(
\begin{array}
[c]{ccc}%
-1 & 0 & 0\\
0 & 1 & 0\\
0 & 0 & 1
\end{array}
\right)  ,\quad\tau:=\left(
\begin{array}
[c]{ccc}%
1 & 0 & 0\\
0 & -1 & 0\\
0 & 0 & 1
\end{array}
\right).
\]
\noindent It has been shown by Francesc \cite{pa23} that, $H\cong
C_2\times C_2$, where $C_m$ denotes the cyclic group of order $m$. Now, computing the fixed points of the automorphisms of $H$ on $C_{a,b,c}$ and their corresponding orbits gives rise to the following result.

\begin{lemma}
For the quartics $C_{a,b,c},$ we have:
\begin{eqnarray*}
X(C_{a,b,c})&=&Orb_{G}[0:\beta:1]\cup Orb_{G}\left[0:\frac{1}{\beta}:1\right]\cup Orb_{G}[\alpha:0:1]\cup\\
&&Orb_{G}[\dfrac{1}{\alpha}:0:1]\cup Orb_{G}[\delta:1:0]\cup Orb_{G}[\dfrac{1}{\delta}:1:0],
\end{eqnarray*}
where $\beta$ is a root of the equation $y^{4}+cx^{2}+1=0,$\,\,\,$\alpha$ is a root of the equation $x^{4}+bx^{2}+1=0$\,\,\, and\,\,\, $\delta$ is a root of the equation $x^{4}+ax^{2}+1=0$.
\end{lemma}

\begin{remark}
\em{Each of the above orbits satisfies $|\,Orb_H(p)\,|=2.$ Moreover, if a point $[\xi:\epsilon:1]\notin X(C_{a,,b,c}),$ then $|\,Orb_H(p)\,|=4.$}
\end{remark}

\begin{proposition}
Let $A_1:=Orb_{G}[0:\beta:1]\cup Orb_{G}\left[0:\frac{1}{\beta}:1\right],$ then for the quartics $C_{a,b,c}$, we have:
\[A_1\cap W_1\left({C_{a,b,c}}\right)\neq \phi\,\,\,\mathrm{iff}\,\,\,\left(a^2+b^2-abc\right)=0.\]
Moreover,
\[A_1\cap W_1\left({C_{a,b,c}}\right)=\left\{
\begin{array}
[c]{lr}%
Orb_{G}[0:\beta:1],\,\,\,\,\,\,\,\,\,\,\,\,\,\,\,\,\,\,\,\,\,\, if\,\,a=\dfrac{1}{2}\left(bc-b\sqrt{c^2-4}\right) & \\\\
Orb_{G}\left[0:\dfrac{1}{\beta}:1\right],\,\,\,\,\,\,\,\,\,\,\,\,\,\,\,\,if\,\,a=\dfrac{1}{2}\left(bc+b\sqrt{c^2-4}\right) & \\\\
\phi\,\,\,\,\,\,\,\,\,\,\,\,\,\,\,\,\,\,\,\,\,\,\,\,\,\,\,\,\,\,\,\,\,\,\,\,\,\,\,\,\,\,\,\,\,\,\,\,\,\,\,\,\,\,\,\,\, otherwise, &
\end{array}
\right.
\]
where $\beta=\dfrac{\sqrt{-c-\sqrt{c^2-4}}}{\sqrt{2}}.$
\end{proposition}

\begin{proof}
The Hessian $H_F(0,y,1)$ is given by the equation \[H_F(0,y,1)=24\left(b+ay^2\right)\left(\left(12-c^2\right)y^2+2c\left(1+y^4\right)\right).\]
So, the resultant of $H_F(0,y,1)$ and $F(0,y,1)$ with respect to $y$ is \[Res\left(H_F(0,y,1),F(0,y,1);y\right)=2985984\left(a^2+b^2-abc\right)^2\left(c^2-4\right)^4,\]
Thus, since $c\neq\pm2,$ \[A_1\cap W_1\left({C_{a,b,c}}\right)\neq \phi\,\,\,\text{iff}\,\,\,\left(a^2+b^2-abc\right)=0.\]
Moreover, the equation $y^4+cy^2+1=0$ has four solutions of the form $\pm\beta,\,\pm\dfrac{1}{\beta},$ where $\beta=\dfrac{\sqrt{-c-\sqrt{c^2-4}}}{\sqrt{2}}.$ Hence, substituting $a=\dfrac{1}{2}\left(bc-b\sqrt{c^2-4}\right)$ into the Hessian equation $H_F(0,y,1)$ yields
\[
H_F[0:\pm\beta:1)=0,\,\,\,\, H_F\left[0:\pm\dfrac{1}{\beta}:1\right]\neq0.
\]
On the other hand, substituting $a=\dfrac{1}{2}\left(bc+b\sqrt{c^2-4}\right)$ into the Hessian equation $H_F(0,y,1)$ yields
\[
H_F[0:\pm\beta:1)\neq0,\,\,\, H_F\left[0:\pm\dfrac{1}{\beta}:1\right]=0,
\]
which completes the proof.
\end{proof}

\begin{proposition}
Let $A_2:=Orb_{G}[\alpha:0:1]\cup Orb_{G}[\dfrac{1}{\alpha}:0:1],$ then for the quartics $C_{a,b,c}$, we have:
\[A_2\cap W_1\left({C_{a,b,c}}\right)\neq \phi\,\,\,\mathrm{iff}\,\,\,\left(a^2+c^2-abc\right)=0.\]
Moreover,
\[A_2\cap W_1\left({C_{a,b,c}}\right)=\left\{
\begin{array}
[c]{lr}%
Orb_{G}[\alpha:0:1],\,\,\,\,\,\,\,\,\,\,\,\,\,\,\,\,\,\,\,\,\,\, if\,\,a=\frac{1}{2} \left(bc-c\sqrt{b^2-4}\right) & \\\\
Orb_{G}[\dfrac{1}{\alpha}:0:1],\,\,\,\,\,\,\,\,\,\,\,\,\,\,\,\,\,\,\,\,if\,\,a=\frac{1}{2} \left(bc+c\sqrt{b^2-4}\right) & \\\\
\phi\,\,\,\,\,\,\,\,\,\,\,\,\,\,\,\,\,\,\,\,\,\,\,\,\,\,\,\,\,\,\,\,\,\,\,\,\,\,\,\,\,\,\,\,\,\,\,\,\,\,\,\,\,\,\,\, otherwise, &
\end{array}
\right.
\]
where $\alpha=\dfrac{\sqrt{-b-\sqrt{b^2-4}}}{\sqrt{2}}$
\end{proposition}

\begin{proof}
The Hessian $H_F(x,0,1)$ is given by the equation \[H_F(x,0,1)=24 \left(a x^2+c\right) \left(-b^2 x^2+12 x^2+2 b
   \left(x^4+1\right)\right.\]
So, the resultant of $H_F(x,0,1)$ and $F(x,0,1)$ with respect to $x$ is \[Res\left(H_F(x,0,1),F(x,0,1);x\right)=2985984\left(a^2+c^2-abc\right)^2\left(b^2-4\right)^4.\]
Hence, since $b\neq\pm2,$ \[A_2\cap W_1\left({C_{a,b,c}}\right)\neq \phi\,\,\,\text{iff}\,\,\,\left(a^2+c^2-abc\right)=0.\]
Moreover, substituting $a=\frac{1}{2} \left(bc-c\sqrt{b^2-4}\right)$ into the Hessian equation $H_F(x,0,1)$ yields
\[
H_F[\pm\alpha:0:1)=0,\,\,\, H_F\left[\pm\dfrac{1}{\alpha}:0:1\right]\neq0.
\]
Also, substituting $a=\frac{1}{2} \left(bc+c\sqrt{b^2-4}\right)$ yields
\[
H_F[\pm\alpha:0:1)\neq0,\,\,\, H_F\left[\pm\dfrac{1}{\alpha}:0:1\right]=0,
\]
we are done.
\end{proof}

\begin{proposition}
Let $A_3:=Orb_{G}[\delta:1:0]\cup Orb_{G}[\dfrac{1}{\delta}:1:0],$ then for the quartics $C_{a,b,c}$, we have:
\[A_3\cap W_1\left({C_{a,b,c}}\right)\neq \phi\,\,\,\mathrm{iff}\,\,\,\left(b^2+c^2-abc\right)=0.\]
Moreover,
\[A_3\cap W_1\left({C_{a,b,c}}\right)=\left\{
\begin{array}
[c]{lr}%
Orb_{G}[\delta:1:0],\,\,\,\,\,\,\,\,\,\,\,\,\,\,\,\,\,\,\,\,\,\, if\,\,a=\frac{1}{2} \left(bc-c\sqrt{b^2-4}\right) & \\\\
Orb_{G}[\dfrac{1}{\delta}:1:0],\,\,\,\,\,\,\,\,\,\,\,\,\,\,\,\,\,\,\,\,if\,\,a=\frac{1}{2} \left(bc+c\sqrt{b^2-4}\right) & \\\\
\phi\,\,\,\,\,\,\,\,\,\,\,\,\,\,\,\,\,\,\,\,\,\,\,\,\,\,\,\,\,\,\,\,\,\,\,\,\,\,\,\,\,\,\,\,\,\,\,\,\,\,\,\,\,\,\,\, otherwise, &
\end{array}
\right.
\]
where $\delta=\dfrac{\sqrt{-a-\sqrt{a^2-4}}}{\sqrt{2}}$
\end{proposition}

\begin{proof}
The Hessian $H_F(x,1,0)$ is given by the equation \[H_F(x,1,0)=24 \left(b X^2+c\right) \left(-a^2 x^2+12 x^2+2 a
   \left(x^4+1\right)\right).\]
So, the resultant of $H_F(x,1,0)$ and $F(x,1,0)$ with respect to $x$ is \[Res\left(H_F(x,1,0),F(x,1,0);x\right)=2985984\left(b^2+c^2-abc\right)^2\left(a^2-4\right)^4 ,\]
Hence, since $a\neq\pm2,$ \[A_3\cap W_1\left({C_{a,b,c}}\right)\neq \phi\,\,\,\text{iff}\,\,\,b^2+c^2-abc=0.\]
Moreover, substituting $b=\frac{1}{2} \left(ac-c\sqrt{a^2-4} \right)$ into the Hessian equation $H_F(x,1,0)$ yields
\[
H_F[\pm\delta:1:0)=0,\,\,\, H_F\left[\pm\dfrac{1}{\delta}:1:0\right]\neq0.
\]
Also, substituting $b=\frac{1}{2} \left(ac+c\sqrt{a^2-4} \right)$ yields
\begin{eqnarray*}
H_F[\pm\delta:1:0)\neq0,\,\,\, H_F\left[\pm\dfrac{1}{\delta}:1:0\right]=0,
\end{eqnarray*}
which completes the proof.
\end{proof}

\begin{proposition}
If $X(C_{a,b,c})\cap W_{1}(C_{a,b,c})\neq\phi$, then the intersection
points are necessarily hyperflex points.
\end{proposition}

\begin{proof}
It suffices to prove the result for the point $[0:\beta:1].$ Indeed, the tangent line is given by \[L_{[0:\beta:1]}: y-\beta=0.\]
The resultant of $L_{[0:\beta:1]}$ and $C_{a,b,c}$ with respect to $y$ is given by \[Res\big(C_{a,b,c},L_{[0:\beta:1]};y\big)=x^2(x^2+a\beta^2+b).\]
Now, substituting $\beta=\dfrac{\sqrt{-c-\sqrt{c^2-4}}}{\sqrt{2}}$ and $a=\dfrac{1}{2}\left(bc-b\sqrt{c^2-4}\right)$ into the last equation yields\[Res\big(C_{a,b,c},L_{[0:\beta:1]};y\big)=x^4.\]
So, $L_{[0:\beta:1]}$ meets $C_{a,b,c}$ at $[0:\beta:1]$ with contact order 4.
\end{proof}
\noindent\textbf{Notations}\begin{itemize}
\item Let $O_{r}$ \ be the orbits classification, where $O$ denotes the number of
orbits and $r$ the number of points in these orbits. For example $2_{4}$
means: two orbits each of 4 points.
\item Let \[P_1(a,b,c):=a^2+b^2-abc,\,\,\,P_2(a,b,c):=a^2+c^2-abc,\,\,\, P_3(a,b,c):=b^2+c^2-abc.\] The set of
common zeros of the equations $P_i(a,b,c)=0$ and $P_j(a,b,c)=0$ such that $i\neq j$ will be denoted by $\Gamma_{ij}.$
   \item Let \[P(a,b):=a^{2}+b^{2}-ab^{2},\text{ \ \ \ }Q(a,b):=36-12a+a^{2}-10b^{2}%
       +3ab^{2}.\] The set of common zeros of the equations $P(a,b)=0\,\,\text{and}\,\,Q(a,b)=0$
will be denoted by $\Gamma.$
                   \end{itemize}
                   \newpage
\noindent Now, we have the following cases:
\begin{description}
  \item[I.] At least two of the parameters vanish,
  \item[II.]The parameters $a,b,c$ satisfy $\Gamma_{ij}$, where $(ij)\in S_2$,
  \item[III.]$(ijk)\in S_3$ such that $P_i(a,b,c)=0$ and $P_j(a,b,c)P_k(a,b,c)\neq0,$
  \item[IV.] If $P_1(a,b,c)P_2(a,b,c)P_3(a,b,c)\neq0$.
\end{description}
The following four theorems treat the above four cases.
\begin{theorem}
Suppose that two of the parameters $a,b,c$ vanish, for instance $b=0,$ $c=0,$ then the number and the geometry of the 1-Weierstrass points for the quartics $C_{a,0,0}$ are classified as follows.  \begin{center}
  \textit{Number and Orbit Classification on $C_{a,0,0}$}%
\end{center}
\[%
\begin{tabular}
[c]{|c|c|c|}\hline
          &\emph{Ordinary flexes} & \emph{Hyperflexes}\\\hline
$a=0,\,6$ &$\mathbf{0}$           & $\mathbf{12}$\\
          &                       & $2_2,\,2_{4}$\\\hline
Otherwise          &$\mathbf{16}$          & $\mathbf{4}$\\
          & $4_{4}$               & $2_{2}$\\\hline
\end{tabular}
\]
where the boldface numbers denote to the number of the points.

\end{theorem}

\begin{proof}
If we assume, without any loss of generality, that $b=0,\,\,c=0$, then the number of the 1-Weierstrass points of $C_{a,0,0}$ is given by the following table \cite{pa20}:
\begin{center}
  \textit{Number Classification on $C_{a,0,0}$}%
\end{center}
\[%
\begin{tabular}
[c]{|c|c|c|}\hline
          &\emph{Ordinary flexes} & \emph{Hyperflexes}\\\hline
$a=0,\,6$ &$\mathbf{0}$           & $\mathbf{12}$\\\hline
 \emph{Otherwise}         &$\mathbf{16}$          & $\mathbf{4}$\\\hline
\end{tabular}
\]
Also, by Proposition 2.2, \[Orb_{G}\left[0:\beta:1\right]\cup Orb_{G}\left[0:\dfrac{1}{\beta}:1\right]\subseteq W_1(C_{a,0,0}).\] Hence, by Proposition 2.5, $C_{0,0,c}$ has $2_2$ of hyperflex points. So by Proposition $2.1(4),$ we have

\begin{center}
  \textit{Orbit Classification on $C_{a,0,0}$}%
\end{center}
\[%
\begin{tabular}
[c]{|c|c|c|}\hline
          &\emph{Ordinary flexes} & \emph{Hyperflexes}\\\hline
$a=0,\,6$ &                       & $2_2,\,2_{4}$\\\hline
Otherwise & $4_{4}$               & $2_{2}$\\\hline
\end{tabular}
\]
\bigskip
This completes the proof.
\end{proof}

\begin{theorem}
If $a,b,c$ satisfy $\Gamma_{ij}$, where $(ij)\in S_2$, then the number and the geometry of the 1-Weierstrass points for the quartics $C_{a,b,c}$ are classified as follows.
\begin{center}
  \textit{Number and Orbit Classification on $C_{a,b,c}$}%
\end{center}

\begin{center}
\begin{tabular}{|c|c|}
  \hline
  Ordinary flex & Hyperflex \\\hline
  $\mathbf{16}$   & $\mathbf{4}$\\
  $4_4$         & $2_2$ \\\hline
  $\mathbf{8}$    & $\mathbf{8}$\\
  $2_4$         & $2_2,\,1_4$ \\\hline
  $\mathbf{0}$    & $\mathbf{12}$\\
   & $2_2,\,2_4$ \\\hline
  \hline
\end{tabular}
\end{center}
\end{theorem}

\begin{proof}
Assume that $a,b,c$\,\, satisfy $\Gamma_{ij}$, where $(ij)\in S_2$, then by Propositions 2.2,\,2.3 and 2.4,\,\, $C_{a,b,c}$ has $2_2$ of flex points. Moreover, by Propositions 2.5, these 2 orbits are hyperflex. Now, recalling that the number of the 1-Weierstrass points of $C_{a,b,c}$, counted with their weights, is $24$. Proposition $2.1(4)$ implies that the orbits of the 1-Weierstrass points for the quartics $C_{a,b,c}$ can be classified as given in the following table:
\begin{center}
  \textit{Orbit Classification on $C_{a,b,c}$}%
\end{center}

\begin{center}
\begin{tabular}{|c|c|}
  \hline
  Ordinary flex & Hyperflex \\\hline
  $4_4$         & $2_2$ \\\hline
  $2_4$         & $2_2,\,1_4$ \\\hline
     & $2_2,\,2_4$ \\\hline
\end{tabular}
\end{center}
\end{proof}

\begin{theorem}
If $(ijk)\in S_3$ such that $P_i(a,b,c)=0$ and $P_j(a,b,c)P_k(a,b,c)\neq0,$ then
the number and the geometry of the 1-Weierstrass points for the quartics $C_{a,b,c}$ are given in the following table:
\begin{center}
  \textit{Number and Orbit Classification on $C_{a,b,c}$}%
\end{center}
\begin{center}
\begin{tabular}{|c|c|}
  \hline
  Ordinary flex & Hyperflex \\\hline
  $\mathbf{20}$& $\mathbf{2}$\\
  $5_4$      & $1_2$ \\\hline
  $\mathbf{12}$& $\mathbf{6}$\\
  $3_4$      & $1_2,\,1_4$ \\\hline
  $\mathbf{4}$ & $\mathbf{10}$\\
  $1_4$ & $1_2,\,2_4$ \\\hline
  \hline
\end{tabular}
\end{center}
\end{theorem}

\begin{proof}
Assume that $P_i(a,b,c)=0$ and $P_j(a,b,c)P_k(a,b,c)\neq0$, where $(ijk)\in S_3.$ Then, $C_{a,b,c}$ has $1_2$ of hyperflex points and the other orbits of the Weierstrass points consist of 4 points. This gives rise to the following table:
\begin{center}
  \textit{Orbit Classification on $C_{a,b,c}$}%
\end{center}
\begin{center}
\begin{tabular}{|c|c|}
  \hline
  Ordinary flex & Hyperflex \\\hline
  $5_4$      & $1_2$ \\\hline
  $3_4$      & $1_2,\,1_4$ \\\hline
  $1_4$ & $1_2,\,2_4$ \\\hline
  \end{tabular}
\end{center}

\end{proof}

\begin{theorem}
If $P_1(a,b,c)P_2(a,b,c)P_3(a,b,c)\neq0$, then the number and the geometry of the 1-Weierstrass points for the quartics $C_{a,b,c}$ are given by the following table:
\begin{center}
  \textit{Number and Orbit Classification on $C_{a,b,c}$}%
\end{center}
\begin{center}
\begin{tabular}{|c|c|}
  \hline
  Ordinary flex & Hyperflex \\\hline
  $\mathbf{24}$  &$\mathbf{0}$\\
  $6_4$          &  \\\hline
  $\mathbf{16}$  &$\mathbf{4}$\\
  $4_4$          & $1_4$ \\\hline
  $\mathbf{8}$   &$\mathbf{8}$\\
   $2_4$         & $2_4$ \\\hline
   $\mathbf{0}$  &$\mathbf{12}$\\
                 & $3_4$ \\\hline
  \hline
\end{tabular}
\end{center}

\end{theorem}

\begin{proof}
Assume that $P_1(a,b,c)P_2(a,b,c)P_3(a,b,c)\neq0$, then $W_1\big(C_{a,b,c}\big)$ is the union of orbits, each consists of four points. So, the geometry of the orbits of the 1-Weierstrass points is given by the following table:
\begin{center}
  \textit{Orbit Classification on $C_{a,b,c}$}%
\end{center}
\begin{center}
\begin{tabular}{|c|c|}
  \hline
  Ordinary flex & Hyperflex \\\hline
  $6_4$          &  \\\hline
  $4_4$          & $1_4$ \\\hline
  $2_4$         & $2_4$ \\\hline
                 & $3_4$ \\\hline
  \end{tabular}
\end{center}
\end{proof}

\section{Examples}
This section is devoted to construct examples that illustrate the cases
mentioned in \emph{Theorems 2.8, 2.9, 2.10}. It should be noted that under a
given  condition more than one case arise, so it is convenient to
investigate whether these cases can occur.\\\\
\textbf{Case II}\hspace*{-.5cm}
\begin{description}
\item[(1)]Let $a=3,\,\,b=3,\,\,c=\frac{3}{2}\left(3-\sqrt{5}\right),$ then $W_1\big(C_{a,b,c}\big)$ consists of $4_4$ of ordinary flex points and $2_2$ of hyperflex points. In other words,
\[W_1\big(C_{a,b,c}\big)=\bigcup_{i=1}^{6}Orb_G[\varsigma_i:\upsilon_i:\epsilon_i],\]
where,
\begin{eqnarray*}
\varsigma_1&:\approx&-0.618034...i,\,\,\,\upsilon_1:=0,\,\,\,\epsilon_1:=1,\\
\varsigma_2&:\approx&-0.618034...i,\,\,\,\upsilon_2:=1,\,\,\,\epsilon_2:=0,\\
\varsigma_3&:\approx&-0.765842...-0.868419...i,\,\,\,\upsilon_3:\approx-0.12504...+0.992152...i,\,\,\,\epsilon_3:=1,\\
\varsigma_4&:\approx&-0.765842...+0.868419...i,\,\,\,\upsilon_4:\approx-0.12504...-0.992152...i,\,\,\,\epsilon_4:=1,\\
\varsigma_5&:\approx&0.521157...-0.432432...i,\,\,\,\upsilon_5:\approx0.184489...+0.982835...i,\,\,\,\epsilon_5:=1,\\
\varsigma_6&:\approx&0.521157...+0.432432...i,\,\,\,\upsilon_6:\approx-0.184489...-0.982835...i\,\,\,\epsilon_6:=1.
\end{eqnarray*}

\item[(2)]Let $c=-\frac{3}{2}\sqrt{\frac{1}{14}\left(21+i\sqrt{7}\right)},\,\,a=\frac{3}{8}\left(5+i \sqrt{7}\right),\,\,b=-\sqrt{\frac{27}{8}+\frac{9 i}{8\sqrt{7}}},$ then $W_1\big(C_{a,b,c}\big)$ consists of 8 hyperflex points and 8 ordinary flex points. In other words,
    \[W_1\big(C_{a,b,c}\big)=\bigcup_{i=1}^{5}Orb_G[\varsigma_i:\upsilon_i:1],\]
where
\begin{eqnarray*}
\varsigma_1&:=&0,\,\,\,\upsilon_1:\approx0.91156...-0.196214...i,\\
\varsigma_2&:\approx&-0.91156...+0.196214...i,\,\,\,\upsilon_2:=0,\\
\varsigma_3&:\approx&0.651994...+0.0981069...i,\,\,\,\upsilon_3:\approx-0.651994...-0.0981069...i,\\
\varsigma_4&:\approx&0.530835...+0.40233...i,\,\,\,\upsilon_4:\approx-1.24444...+0.212546...i,\\
\varsigma_5&:\approx&-1.24444...+0.212546...i,\,\,\,\upsilon_5:\approx0.530835...+0.40233...i.
\end{eqnarray*}
\item[(3)]Let $c=\frac{6}{\sqrt{5}},\,\,a=\frac{6}{5},\,\,b=\frac{6}{\sqrt{5}},$ then $W_1$ consists of 12 hyperflex points \cite{pa2}.
\end{description}
\newpage
\textbf{Case III}
\begin{description}
  \item[(1)] Let $c=4,b=6-3\sqrt{3},a=3$, then $W_1\big(C_{a,b,c}\big)$ consists of $22$ flex points. In other words,
      \[W_1\big(C_{a,b,c}\big)=\bigcup_{i=1}^{6}Orb_G[\varsigma_i:\upsilon_i:1],\]
      where
\begin{eqnarray*}
\varsigma_1&:\approx&3.72978...,\,\,\,\upsilon_1:\approx2.2488...i,\\
\varsigma_2&:\approx&0.225851...+1.28153...i,\,\,\,\upsilon_2:\approx1.14986...-1.07474...i,\\
\varsigma_3&:\approx&0.225851...-1.28153...i,\,\,\,\upsilon_3:\approx1.14986...+1.07474...i,\\
\varsigma_4&:\approx&0.334413...-1.0111...i,\,\,\,\upsilon_4:\approx-0.471629...+0.349376...i,\\
\varsigma_5&:\approx&0.334413...+1.0111...i,\,\,\,\upsilon_5:\approx-0.471629...-0.349376...i,\\
\varsigma_6&:=&0,\,\,\,\upsilon_6:\approx0.517638...i.
\end{eqnarray*}
      \end{description}
\begin{remark}
\em{There is no guarantee that the other two cases for which $W_1\big(C_{a,b,c}\big)$ consists of $18$ or $14$ flexes can occur.}
\end{remark}
\noindent\textbf{Case IV}
\begin{description}
  \item[(1)] Let $a=3,\,\,b=3,\,\,c=0,$ then $C_{a,b,c}$ has $24$ ordinary flex points. In other words;
  \[W_1(C)=\bigcup_{i=1}^{6}Orb_G[\varsigma_i:\upsilon_i:1],\]
  where

\begin{eqnarray*}
\varsigma_1&:=&0-1.75642...i,\,\,\,\upsilon_1:=-3.01936...,\\
\varsigma_2&:\approx&-0.581718...i,\,\,\,\upsilon_2:\approx-0.33119...,\\
\varsigma_3&:\approx&-0.91777...+1.15085...i,\,\,\,\upsilon_3:\approx-0.22252...-0.97492...i,\\
\varsigma_4&:\approx&-0.91777...-1.15085...i,\,\,\,\upsilon_4:\approx-0.22252...+0.97492...i,\\
\varsigma_5&:\approx&-0.59367...+0.39822...i,\,\,\,\upsilon_5:\approx-0.37935...+0.92525...i,\\
\varsigma_6&:\approx&-0.593675...-0.39822...i,\,\,\,\upsilon_6:\approx-0.37935...-0.92525...i.
\end{eqnarray*}
\item[(2)] Let $a=\sqrt{5}i,\,\,b=3,\,\,c=0,$ then
\[W_1(C)=\bigcup_{i=1}^{4}Orb_H[i \sqrt{\frac{2}{3}}:\upsilon_i:1],\]
  where $\upsilon_i$ are the four roots of the equation $x^4+1=0.$ That is, $C_{a,b,c}$ has $16$ flex points classified as $2_2$ of hyperflex and $2_2$ of ordinary flex.
\item[(3)] Let $c=3\sqrt{\frac{2}{5}},\,\,b=3\sqrt{\frac{2}{5}},\,\,a=0,$ then
   \[W_1=\bigcup_{i=0}^{4}Orb_H[\xi_i:\epsilon_i:1],\]
   where
\begin{eqnarray*}
\xi_0&:\approx&0.562341i,\,\,\,\epsilon_0:\approx0.562341i,\\
\xi_1&:\approx&0.204102-1.15107i,\,\,\,\epsilon_1:\approx0.301675-0.269467i,\\
\xi_2&:\approx&0.301675-0.269467i,\,\,\,\epsilon_2:\approx0.204102-1.15107i,\\
\xi_3&:\approx&0.204102 + 1.15107i,\,\,\,\epsilon_3:\approx0.301675 + 0.269467i,\\
\xi_4&:\approx&0.301675 + 0.269467i,\,\,\,\epsilon_4:\approx0.204102 + 1.15107i.
\end{eqnarray*}
\item[(4)]$a=3,\,\,b=3\,\,c=3,$ then $C_{a,b,c}$ has $12$ hyperflex points \cite{pa20}.
\end{description}

\noindent{\fontsize{13.5}{9}\textbf{Concluding remarks.}}\\
We conclude the present paper with some remarks and comments.
\begin{itemize}
\item The computations included in this work have been performed by the use of MATHEMATICA program. The source code files are available.
\item The classification of the 1-Weierstrass points of \emph{Kuribayashi quartics with two parameters} treated in \cite{pa20} is a particular case of our results. In fact, letting $b=c$, one gets the following tables.
\begin{center}
  \textit{Case $b=0$}%
\end{center}
\[%
\begin{tabular}
[c]{|c|c|c|}\hline
\emph{Ordinary flexes} & \emph{Hyperflexes}\\\hline
$\mathbf{0}$           & $\mathbf{12}$\\
 $\mathbf{16}$          & $\mathbf{4}$\\\hline
\end{tabular}
\]

\begin{center}
   \textit{Case $b\neq0$}%
\end{center}
\begin{center}
\begin{tabular}{|c|c|c|}
  \hline
  & Ordinary flex & Hyperflex \\\hline
  &$\mathbf{16}$   & $\mathbf{4}$\\
  $P(a,b)=0$&$\mathbf{8}$    & $\mathbf{8}$\\
  &$\mathbf{0}$    & $\mathbf{12}$\\\hline
  &$\mathbf{24}$  &$\mathbf{0}$\\
  $P(a,b)\neq0$&$\mathbf{16}$  &$\mathbf{4}$\\
  &$\mathbf{8}$   &$\mathbf{8}$\\
  &$\mathbf{0}$  &$\mathbf{12}$\\\hline
\end{tabular}
\end{center}
\item The main theorems constitute a motivation to solve more general problems. One of these problems is the  investigation of the geometry of higher order and multiple Weierstrass points of $C_{a,b}$ to generalize the classification of the 2-Weierstrass points of \emph{Kuribayashi quartics with one parameter family} \cite{pa3, pa4}. However, this problem will be the object of a forthcoming work.
\item The technique used in this paper is completely different from that used by Hayakawa \cite{pa5}. Our technique consists of dividing the quartics by group actions into finite orbits and investigate the geometry of these orbits. The results obtained are more informative than those obtained by Hayakawa since we give the geometry of the 1-Weierstrass points.
\end{itemize}
\textbf{Acknowledgment} The authors would like express their sincere
gratitude to Prof. Nabil L. Youssef for his guidance throughout the
preparation of this work.

\bigskip\noindent

\end{document}